\documentclass[11pt]{amsart}
\usepackage{mathptmx, amssymb, amsmath, mathtools, enumerate, colonequals}

\RequirePackage[dvipsnames,usenames]{xcolor}

\usepackage[margin=1.2in]{geometry}
\definecolor{chianti}{rgb}{0.6,0,0}
\definecolor{meretale}{rgb}{0,0,.6}
\definecolor{leaf}{rgb}{0,.35,0}
\usepackage[colorlinks=true, pagebackref, hyperindex, citecolor=meretale, urlcolor=leaf, linkcolor=chianti]{hyperref}

\newtheorem{theorem}{Theorem}[section]
\newtheorem{corollary}[theorem]{Corollary}

\newtheorem{proposition}[theorem]{Proposition}

\theoremstyle{definition}

\newtheorem{remark}[theorem]{Remark}
\newtheorem{question}[theorem]{Question}

\numberwithin{equation}{theorem}

\AtBeginEnvironment{remark}{\pushQED{\qed}}%
  \AtEndEnvironment{remark}{\popQED}

\newcommand{\height}{\operatorname{ht}}

\renewcommand{\to}{\longrightarrow}

\newcommand{\frakm}{\mathfrak{m}}

\newcommand{\FF}{\mathbb{F}}

\newcommand{\QQ}{\mathbb{Q}}
\newcommand{\ZZ}{\mathbb{Z}}

\title{Remarks on effective uniform Brian\c{c}on-Skoda}

\author{Alexandria Wheeler and Wenliang Zhang}
\address{Department of Mathematics, Statistics, and Computer Science, University of Illinois at Chicago,
Chicago, IL 60607}
\email{awheel23@uic.edu, wlzhang@uic.edu}
\thanks{The authors acknowledge support from NSF through the grant DMS-1752081.}
\begin{document}

\maketitle

\begin{abstract}
Let $R$ be a noetherian commutative ring. Of great interest is the question whether one can find an explicit integer $k$ such that $\overline{I^{k+n}}\subseteq I^n$ for each ideal $I$ and each integer $n\geq 1$ (the notation $\overline{I^{k+n}}$ denotes the integral closure of $I^{k+n}$). In this article, we investigate this question and obtain optimal values of $k$ for $F$-pure (or dense $F$-pure type) rings and Cohen-Macaulay $F$-injective (or dense $F$-injective type) rings. 
\end{abstract}


\section{Introduction}

In the seminal paper \cite{HunekeUniformBounds}, the following theorem is proved.
\begin{theorem}[Huneke]
\label{thm: huneke}
Let $R$ be a noetherian reduced ring. Assume that $R$ satisfies one of the following conditions:
\begin{enumerate}
\item $R$ is essentially of finite type over a noetherian local ring;
\item $R$ is of prime characteristic $p$ and is a finitely generated $R^p$-module; or
\item $R$ is essentially of finite type over $\ZZ$.
\end{enumerate}
Then there exists an integer $k$ such that $\overline{I^{k+n}}\subseteq I^{n}$ for every ideal $I$ and every integer $n\geq 1$.
\end{theorem}
This theorem is called a `uniform Brian\c{c}on-Skoda' in \cite[Theorem~4.13]{HunekeUniformBounds}. The notation $\overline{I^{k+n}}$ denotes the integral closure of $I^{k+n}$, which can be defined as follows: let $R$ be a noetherian ring and $J$ be an ideal. An element $r\in R$ is \emph{integral over} $J$ if there exist an integer $m$ and elements $a_i\in J^i$ such that
\[r^m+a_1r^{m-1}+\cdots+a_{m-1}r+a_m=0.\]
\cite[Remark~4.14]{HunekeUniformBounds} remarks that ``\emph{It is of great interest to find the best possible $k$ in Theorem \ref{thm: huneke} in terms of invariants of the ring $R$. In fact it is of major interest in the case where $R$ is local.}'' In the spirit of this remark, we ask the following question:

\begin{question}
\label{main question}
Let $R$ be a noetherian local ring. Can one find an explicit value of $k$ such that
\[\overline{I^{k+n}}\subseteq I^{n}\]
for every ideal $I$ and every integer $n\geq 1$?
\end{question} 

Affirmative answers to Question \ref{main question} are known in some cases. The first known case is proved in \cite{BrianconSkoda} when $R$ is the ring of convergent power series over the complex numbers and then it is proved in \cite{LipmanTeissier} that if $R$ is a pseudo-rational local ring then 
\[\overline{I^{n+\dim(R)-1}}\subseteq I^n\]
for each integer $n\geq 1$. That is, when $R$ is a pseudo-rational local ring, one may take $k=\dim(R)-1$.

The motivation behind this article is to find explicit values for $k$ beyond the pseudo-rational case. One of our main results is the following:

\begin{theorem}[Theorems \ref{prop: integral and Frob closure}, \ref{uniform Skoda for I}, \ref{uniform BS for CM with infinite residue field} and Corollary \ref{uniform BS for CM}]
\label{main thm in intro}
Let $R$ be a noetherian local ring of prime characteristic $p$. For each ideal $I$, let $\ell(I)$ denote the analytic spread of $I$ and $\height(I)$ denote the height of $I$. 
\begin{enumerate}
\item Assume that $R$ is $F$-pure. Then 
\[\overline{I^{\dim(R)+n}}\subseteq I^n\] 
for every ideal $I$ and each integer $n\geq 1$.\\
If the residue field of $R$ is infinite and $R$ is $F$-pure, then 
\begin{equation}
\label{intro: Skoda for F pure}
\overline{I^{\ell(I)+n}}\subseteq I^n,\quad \forall n\geq 1.
\end{equation}
for each ideal $I$. 
\vspace{0.1cm}
\item Assume that $R$ is equidimensional and catenary; that every parameter ideal in $R$ is Frobenius closed (e.g. generalized Cohen-Macaulay and $F$-injective); and that the residue field is infinite. Then  
\begin{equation}
\label{intro Skoda for parameter}
\overline{I^{2\ell(I)-\height(I)+1}}\subseteq I
\end{equation}
for each ideal $I$. In particular, $\overline{I^{\dim(R)+1}}\subseteq I$ for each ideal $I$.
\vspace{0.1cm}
\item Assume that $R$ is Cohen-Macaulay and $F$-injective. Then 
\[\overline{I^{\dim(R)+n}}\subseteq I^n\] 
for each ideal $I$ and each integer $n\geq 1$.\\
If $R$ is Cohen-Macaulay, every parameter ideal in $R$ is Frobenius closed and the residue field of $R$ is inifnite, then
\begin{equation}
\label{intro Skoda for parameter CM}
\overline{I^{2\ell(I)-\height(I)+n}}\subseteq I^n
\end{equation}
for each ideal $I$ and each integer $n\geq 1$.
\end{enumerate}
\end{theorem} 

The exponents $\ell(I)+n$ in (\ref{intro: Skoda for F pure}), $2\ell(I)-\height(I)+1$ in (\ref{intro Skoda for parameter}) and $2\ell(I)-\height(I)+n$ in (\ref{intro Skoda for parameter CM}) are optimal ({\it cf.} Remarks \ref{bound optimal for F-pure} and \ref{rmk: optimal for F-closed}). 

The article is organized as follows. In \S\ref{section: prelim}, we collect some necessary background materials; in \S\ref{section: main results in char p} we investigate Question \ref{main question}, Theorem \ref{main thm in intro} is proved in this section; in \S\ref{results char 0} we consider the counterparts of Theorem \ref{main thm in intro}.

\section*{Acknowledgement} The authors thank Craig Huneke and Karl Schwede for their comments on a preliminary version of the article.

\section{Preliminaries}
\label{section: prelim}
In this section, we collect some materials necessary for the rest of this article. We begin by recalling some basic facts on reduction and analytic spread of ideals.
\begin{remark}
\label{rmk: basic on reduction and analytic spread}
Let $R$ be a noetherian ring and $I$ be an ideal. We will denote by $R^\circ$ the set of elements in $R$ not contained in any minimal prime ideal of $R$.

\begin{enumerate}
\item An element $r$ is integral over $I$ if there is an element $c\in R^\circ$ such that $cr^n\in J^n$ for $n\gg 0$. 
\item An ideal $J\subseteq I$ is called a {\it reduction} of $I$ if there is an integer $t$ such that $JI^t=I^{t+1}$. A reduction $J$ of $I$ is called {\it minimal} if $J$ is minimal (with respect to inclusion) among all reductions of $I$. If $(R,\frakm)$ is a local ring with an infinite residue field, every ideal admits a minimal reduction. 
\item For each ideal $I$ in a local ring $(R,\frakm)$, we set
\[\ell(I):=\dim(\bigoplus_{j\geq 0} \frac{I^j}{\frakm I^j})\]
where $\dim$ denotes the Krull dimension. The quantity $\ell(I)$ is called the {\it analytic spread} of $I$. It is known that $\height(I)\leq \ell(I)\leq \dim(R)$. ({\it cf.} \cite[\S~8.3]{HunekeSwansonBook})
\item If $(R,\frakm)$ is a local ring with an infinite residue field, every ideal $I$ admits a minimal reduction generated by $\ell(I)$ elements. ({\it cf.} \cite[\S~8.4]{HunekeSwansonBook})
\item There is a classic construction of a flat local extension $R(X)$ of a local ring $(R,\frakm)$ such that the maximal ideal of $R(X)$ is $\frakm R(X)$ and $R(X)/\frakm R(X)$ is infinite. More specifically, following the notation in \cite[\S~8.4]{HunekeSwansonBook}, consider the polynomial ring $R[X]$ in one variable $X$ and set 
\[R(X):=R[X]_{\frakm R[x]}\]
Then $R(X)$ is a flat local extension of $R$, the unique maximal ideal in $R(X)$ is $\frakm R(X)$, and $R(X)/\frakm R(X)$ is the field of fractions of $(R/\frakm)[X]$ and hence is infinite. Note that $R$ and $R(X)$ has the same Krull dimension, that $R$ is equidimensional if and only if so is $R(X)$, and that $R$ is Cohen-Macaulay if and only if so is $R(S)$. ({\it cf.} \cite[\S~8.4]{HunekeSwansonBook})
\end{enumerate}
\end{remark}

We observe that one can reduce Question \ref{main question} to the case when $I$ is $\frakm$-primary.
\begin{remark}
\label{rem: reduce to primary}
Let $(R,\frakm)$ be a noetherian local ring. Assume that there is an integer $k$ such that $\overline{J^{n+k}}\subseteq J^n$ for all $\frakm$-primary ideals $J$, then $\overline{I^{n+k}}\subseteq I^n$ for all ideals $I$. To see this, we simply observe that for each integer $N\geq 1$
\[\overline{I^{n+k}}\subseteq \overline{(I+\frakm^N)^{n+k}}\subseteq (I+\frakm^N)^n \subseteq I^n+\frakm^N,\]
where the second inclusion holds since $I+\frakm^N$ is $\frakm$-primary. Consequently 
\[\overline{I^{n+k}}\subseteq \bigcap_N I^n+\frakm^N=I^n.\]
\end{remark}

We recall some basic notions in prime characteristic $p$. Let $R$ be a noetherian rings of prime characteristic $p$ and $I$ be an ideal. 
\begin{itemize}
\item For each integer $e$, the ideal generated by $\{a^{p^e}\mid a\in I\}$ is denoted by $I^{[p^e]}$. 
\item An element $r\in R$ is in the {\it Frobenius closure} of $I$ if there is an integer $e$ such that $r^{p^e}\in I^{[p^e]}$. The Frobenius closure of $I$ will be denoted by $I^F$; an ideal $I$ is called {\it Frobenius closed} if $I^F=I$. 
\item $R$ is called {\it $F$-pure} if the Frobenius morphism $\varphi: R\xrightarrow{r\mapsto r^p}R$ is a pure morphism; that is $R\otimes_RM\xrightarrow{\varphi\otimes{\bf 1}}R\otimes_RM$ is injective for every $R$-module $M$. 
\item An $F$-pure ring is reduced and for every ideal is Frobenius closed in an $F$-pure ring.
\end{itemize}

\section{Results in prime characteristic $p$}
\label{section: main results in char p}

We begin with an effective uniform Brian\c{c}on-Skoda theorem for $F$-pure rings ({\it cf.} \cite[Theorem~2.2]{KatzmanZhangMultiplicity} for the case when $n=1$).

\begin{theorem}
\label{prop: integral and Frob closure}
Let $R$ be a noetherian rings of prime characteristic $p$. Assume that each $c\in R^\circ$ is a non-zero-divisor ({\it e.g.} $R$ is reduced). If $I$ is an ideal that can be generated by $\ell$ elements, then 
\[\overline{I^{\ell+n}}\subseteq (I^n)^F, \quad \forall n\geq 1.\]
In particular, if $R$ is $F$-pure, then
\[\overline{I^{\ell+n}}\subseteq I^n.\]
\end{theorem}
\begin{proof}
Given an arbitrary element $x\in \overline{I^{\ell+n}}$, there exists $c\in R^\circ$ such that $cx^t\in (I^{\ell+n})^t$ for all $t\gg 0$. That is, $x^t\in (I^{(\ell+n)t}:c)$. It follows that
\[cx^t\in c(I^{(\ell+n)t}:c)\subseteq (c)\cap I^{(\ell+n)t}\subseteq cI^{(\ell+n)t-m}\]
where the last inclusion follows from the Artin-Rees Lemma and $m$ is an integer independent of $t$. Since $c$ is a non-zero-divisor by our assumption, we have 
\[x^t\in I^{(\ell+n)t-m},\quad \forall n\gg 0.\]

Writing $I=(g_1,\dots,g_{\ell})$ and setting $t=q=p^e\gg 0$, we have $x^q\in I^{(\ell+n)q-m}$. We claim that $I^{(\ell+n)q-m}\subseteq (I^n)^{[q]}$ and we reason as follows. It suffices to show that each generator $g^{a_1}_1\cdots g^{a_{\ell}}_{\ell}$ ($\sum_i a_i=(\ell+n)q-m$) belongs to $(I^n)^{[q]}$. Write $a_i=s_iq+t_i$ where $s_i\geq 0$ and $0\leq t_i\leq q-1$. If $\sum_i s_i\leq n-1$, then 
\[\sum_i a_i=(\sum_i s_i)q+\sum_i t_i\leq (n-1)q+\ell(q-1)=(\ell+n)q-q-\ell< (\ell+n)q-m,\]
a contradiction. Therefore, $\sum_i s_i\geq  n$ and consequently 
\[I^{(\ell+n)q-m}\subseteq (I^n)^{[q]}.\]
Thus, $x^q\in (I^n)^{[q]}$; that is, $x\in (I^n)^F$. This completes that proof of $\overline{I^{\ell+n}}\subseteq (I^n)^F$ for all $n\geq 1$.
\end{proof}

\begin{corollary}
\label{Skoda for F-pure with generator}
Let $(R,\frakm)$ be an $F$-pure noetherian local ring of prime characteristic $p$ with an infinite residue field. For each ideal $I$, let $\ell(I)$ denotes its analytic spread of $I$. Then
\[\overline{I^{\ell(I)+n}}\subseteq I^n,\quad \forall n\geq 1.\]
\end{corollary}
\begin{proof}
Since the residue field of $R$ is infinite, $I$ admits a reduction $J$ with $\ell(I)$ generators. Replacing $I$ by $J$, one may assume that $I$ is generated by $\ell(I)$ elements. The corollary follows from Theorem \ref{prop: integral and Frob closure}.
\end{proof}

\begin{corollary}
\label{Skoda for F-pure}
Let $(R,\frakm)$ be an $F$-pure noetherian local ring of prime characteristic $p$. Then 
\[\overline{I^{\dim(R)+n}}\subseteq I^n\] 
for each integer $n\geq 1$.
\end{corollary}
\begin{proof}
If the residue field of $R$ is infinite, then this follows immediately from Corollary \ref{Skoda for F-pure with generator}. Otherwise, consider the flat local extension $R\to R(X)$ as in Remark \ref{rmk: basic on reduction and analytic spread}. Let $I$ be an ideal in $R$. Since the residue field of $R(X)$ is infinite, $IR(X)$ admits a minimal reduction $J$ with $\ell(IR(X))$ elements. Replacing $IR(X)$ by $J$, we may assume that $IR(X)$ is generated by $\ell(IR(X))$ elements. Note that $\ell(IR(X))\leq \dim(R(X))=\dim(R)$. Since the closed fiber of $R\to R(X)$ is a field, $R(X)$ is $F$-pure as well. Hence it follows from Corollary \ref{Skoda for F-pure with generator} that 
\[\overline{(IR(X))^{\ell(I)+n}}\subseteq (IR(X))^n.\]
In particular
 \[\overline{(IR(X))^{\dim(R)+n}}=\overline{(IR(X))^{\dim(R(X))+n}}\subseteq (IR(X))^n,\quad \forall n\geq 1.\]
Consequently,
\[\overline{I^{\dim(R)+n}}\subseteq \overline{(IR(X))^{\dim(R)+n}}\bigcap R\subseteq (IR(X))^n\bigcap R=I^n,\quad \forall n\geq 1\]
where the last equality follows from the fact $R(X)$ is a faithfully flat extension of $R$.
 \end{proof}

\begin{remark}
\label{bound optimal for F-pure}
For $F$-pure rings, $\dim(R)$ is the optimal value for $k$ in Theorem \ref{thm: huneke}. 

Consider $R=\overline{\FF_p}[[x,y,x]]/(x^3+y^3+z^3)$ with $p\equiv 1$ (mod 3). Then $R$ is Cohen-Macaulay and $F$-pure; hence the hypotheses in Corollary \ref{Skoda for F-pure} are satisfied. Set $I=(x,y)$. Then $\overline{I^3}\subseteq I$ according to Corollary \ref{Skoda for F-pure}, but $\overline{I^2}\not\subseteq I$ ( $z^2\in \overline{I^2}$ since $(z^2)^3\in (I^2)^3$, but $z^2\notin I$). 
\end{remark}

To extend Theorem \ref{Skoda for F-pure} beyond $F$-pure rings, we will consider rings in which every parameter ideal is Frobenius closed. One approach is to ``approximate'' a given ideal with a parameter ideal, using \cite[3.1,~3.2]{Aberbach-Huneke}.

\begin{remark}
\label{approximate ideal use AH}
Let $(R,\frakm)$ be a noetherian local ring with an infinite residue field and $I$ be an ideal. Let $J$ be a minimal reduction with $\ell=\ell(I)$ generators ($\ell(I)$ denotes the analytic spread of $I$). Then \cite[3.1]{Aberbach-Huneke} shows that there exists a generating set $\{a'_1,\dots,a'_{\ell}\}$ for $J$ such that
\begin{enumerate}
\item If $P$ is a prime ideal containing $I$ and $\height(P)=i\leq \ell(I)$, then $(a'_1,\dots,a'_i)_P$ is a reduction of $IR_P$, and
\item $\height(((a'_1,\dots,a'_i)I^n:I^{n+1})+I)\geq i+1$ for all $n\gg 0$.
\item If $c_i\equiv a'_i$ modulo $I^2$, then $(1)$ and $(2)$ hold with $c_i$ replacing $a'_i$.
\end{enumerate}
Fix such a generating set $\{a'_1,\dots,a'_{\ell}\}$ for $J$ and let $h$ denote the height of $I$. \cite[3.2]{Aberbach-Huneke} shows that for nonnegative integers $N$ and $w$  there are elements $a_1,\dots,a_{\ell}$ and $t_{h+1},\dots,t_\ell$ (we set $t_i=0$ for $i\leq h$) such that the following hold:
\begin{enumerate}
\item $a_i\equiv a'_i$ modulo $I^2$;
\item $t_i\in \frakm^N$ for $h+1\leq i\leq \ell$;
\item $b_1,\dots,b_{\ell}$ are parameters, where $b_i=a_1+t_i$;
\item there is an integer $M$ such that $t_{i+1}I^{M+s}\subseteq J^s_iI^M$ for all $0\leq s\leq w+\ell$, where $J_i=(a_1,\dots,a_i)$.
\end{enumerate}
\end{remark}

\begin{theorem}
\label{thm: inclusion for Frobenius closure}
Let $(R,\frakm)$ be an equidimensional and catenary reduced local ring of characteristic $p$ having infinite residue field. Let $I$ be an ideal of analytic spread $\ell$ and height $h$. Let $J$ be a minimal reduction of $I$. Fix integers $w\geq \ell-h+1,N\geq 0$. Choose $a_i$, $t_i$, $J_i$, and $M$ as in Remark \ref{approximate ideal use AH}. Set $B=(b_1,\dots,b_{\ell})$. Assume that each element in $R^{\circ}$ is a nonzerodivisor. Then
\[\overline{I^{\ell+w}}\subseteq (B^{w-(\ell-h)})^F.\]
\end{theorem}
\begin{proof}
If $I=(0)$, then there is nothing to prove. Assume that $I\neq (0)$.

By the choice of $t_j$, we have $t_jI^{M+s}\subseteq J^s_{j-1}I^M$ and consequently 
\[t^n_jI^{M+ns}\subseteq J^{ns}_{j-1}I^M\]
by an induction on $n$. 

If $h>0$, then $I^M\cap R^{\circ}\neq \emptyset $. In this case, pick an element $c\in I^M\cap R^{\circ}$. If $h=0$, then $I$ is not contained in all the minimal primes. By Prime Avoidance, one can pick $c_1\in I^M$ that is not contained in any minimal prime that does not contain $I$ and $c_2\notin I$ to be an element in the intersection of the minimal primes not containing $I$. Note that $c_2I=0$. In this case, set $c:=c_1+c_2$. Then $c\in R^{\circ}$.

By the choice of $c$, one has
\begin{equation}
\label{equ: I J inclusion}
ct^{n}_jI^{ns}\subseteq J^{ns}_{j-1}.
\end{equation}
Set $B_i=(b_1,\dots,b_i)$ for $h\leq i\leq \ell$. We wish to show by induction on $i$ that, for each $0\leq r\leq w$ and each $q=p^e$, 
\begin{equation}
\label{equ: JB inclusion}
c^{i-h}J^{(i+r)q}_i\subseteq B^{(h+r)q}_i.
\end{equation}
The base case is when $i=h$ and $0\leq r\leq w$ is arbitrary. Since $J_h=B_h$, the base case is true.

Assume now that we are given $r$ and $i>h$ and that (\ref{equ: JB inclusion}) is true either for $i'<i$ (with $r'\leq w$ arbitrary) or for $i'=i$ (with $r'<r\leq w$). By our choice of $c$ and of $t_j$,
\begin{align*}
&\quad c^{i-h}J^{(i+r)q}_i\\
&= c^{i-h}J^{[q]}_iJ^{(i+r-1)q}_i \ {\rm (since\ } J_i\ {\rm can\ be\ generated\ by\ } i\ {\rm elements)}\\
&=c^{i-h}\Big[ J^{[q]}_hJ^{(i+r-1)q}_i + a^q_{h+1}J^{(i+r-1)q}_i+\cdots +a^q_iJ^{(i+r-1)q}_i\Big]\\
&=c^{i-h}\Big[ J^{[q]}_hJ^{(i+r-1)q}_i + (b^q_{h+1}-t^q_{h+1})J^{(i+r-1)q}_i+\cdots +(b^q_{i}-t^q_{i})J^{(i+r-1)q}_i\Big]\\
&=c^{i-h-1}\Big[ cJ^{[q]}_hJ^{(i+r-1)q}_i + c(b^q_{h+1}-t^q_{h+1})J^{(i+r-1)q}_i+\cdots +c(b^q_{i}-t^q_{i})J^{(i+r-1)q}_i\Big]\\
&\subseteq c^{i-h-1}\Big[ cJ^{[q]}_hJ^{(i+r-1)q}_i + (cb^q_{h+1}J^{(i+r-1)q}_i+ct^q_{h+1}I^{(i+r-1)q})+\cdots +(cb^q_{i}J^{(i+r-1)q}_i+ct^q_{i}I^{(i+r-1)q})\Big]\\
&\stackrel{(\ref{equ: I J inclusion})}{\subseteq} c^{i-h-1}\Big[ cJ^{[q]}_hJ^{(i+r-1)q}_i + (cb^q_{h+1}J^{(i+r-1)q}_i+J^{(i+r-1)q}_{h})+\cdots +(cb^q_{i}J^{(i+r-1)q}_i+J^{(i+r-1)q}_{i-1})\Big]\\
&\stackrel{(\ref{equ: JB inclusion})}{\subseteq} B^{(h+r)q}_i\ {\rm since\ by\ our\ induction\ hypothesis\ each\ term\ is\ in\ }B^{(h+r)q}_i.
\end{align*}
Therefore, 
\[c^{\ell-h}J^{(\ell+r)q}=c^{\ell-h}J^{(\ell+r)q}_{\ell}\subseteq B^{(h+r)q}_\ell=B^{(h+r)q}.\]

Now given each $z\in \overline{I^{\ell+w}}=\overline{J^{\ell+w}}$, there exists a $d\in R^{\circ}$ such that $dz^q\in J^{(\ell+w)q}$. Hence  
\[c^{\ell-h}dz^q\in c^{\ell-h}J^{(\ell+w)q}\subseteq B^{(h+w)q}\]
where the last inclusion holds since $w\geq \ell-h+1$. Then $z^q\in (B^{(h+w)q}:(c^{\ell-h}d))$. (At this point, we follow the same strategy as in the proof of Proposition \ref{prop: integral and Frob closure}.) It follows that
\[c^{\ell-h}dz^q\in (c^{\ell-h}d)(B^{(h+w)q}:(c^{\ell-h}d))\subseteq (c^{\ell-h}d)\cap B^{(h+w)q}\subseteq (c^{\ell-h}d)B^{(h+w)q-m}\]
for a fixed $m$ (independent of $q$) where the last inclusion holds due to the Artin-Rees Lemma. Since $w\geq \ell-h+1$, we have $B^{(h+w)q-m}\subseteq (B^{w-\ell+h})^{[q]}$ for all $q\gg 0$ since $B$ is generated by $\ell$ elements. Since $c^{\ell-h}d$ is a nonzerodivisor, it follows that $z^q\in (B^{w-\ell+h})^{[q]}$; consequently $z\in (B^{w-\ell+h})^F$.
\end{proof}


\begin{theorem}
\label{uniform Skoda for I}
Let $(R,\frakm)$ be an equidimensional and catenary local ring of characteristic $p$ with an infinite residue field. Let $I$ be an ideal of analytic spread $\ell(I)$ and height $\height(I)$. Assume that each parameter ideal in $R$ is Frobenius closed. Then
\begin{equation}
\label{Skoda with analytic spread}
\overline{I^{2\ell(I)-\height(I)+1}}\subseteq I.
\end{equation}
In particular, $\overline{I^{\dim(R)+1}}\subseteq I$ for each ideal $I$.
\end{theorem}
\begin{proof}
Since each parameter ideal in $R$ is Frobenius closed, $R$ is reduced; hence each element in $R^{\circ}$ is a nonzerodivisor; the hypothesis in Theorem \ref{thm: inclusion for Frobenius closure} are satisfied. Note that the ideal $B$ as in Theorem \ref{thm: inclusion for Frobenius closure} is a parameter ideal and hence Frobenius closed by the assumption. It follows from Theorem \ref{thm: inclusion for Frobenius closure} (with $w=\ell(I)-\height(I)+1$) that
\[\overline{I^{2\ell(I)-\height(I)+1}}\subseteq B.\]
By the construction of $B$, we have $B\subseteq J+\frakm^N$. Therefore
\[\overline{I^{2\ell(I)-\height(I)+1}}\subseteq J+\frakm^N\quad \forall N\]
It follows from the Krull Intersection Theorem that 
\[\overline{I^{2\ell(I)-\height(I)+1}}\subseteq J\subseteq I.\]

It remains to show that $\overline{I^{\dim(R)+1}}\subseteq I$ for each ideal $I$. According to Remark \ref{rem: reduce to primary}, we may assume that $I$ is primary to the maximal ideal. Then $\height(I)=\dim(R)$. Since $\ell(I)\leq \dim(R)$. The desired inclusion follows.
\end{proof}

\begin{remark}
\label{rmk: optimal for F-closed}
\begin{enumerate}
\item There are examples of non-$F$-pure rings in which every parameter ideals is Frobenius closed ({\it cf.} \cite[\S6]{QuyShimomoto}).
 
\item The exponent $2\ell(I)-\height(I)+1$ in (\ref{Skoda with analytic spread}) is optimal. Consider $R=\overline{\FF_p}[[x,y,x]]/(x^3+y^3+z^3)$ with $p\equiv 1$ (mod 3) and $I=(x,y)$. Then $\ell=\ell(I)=2$ and $\height(I)=2$; hence $2\ell(I)-\height(I)+1$=3. According to (\ref{Skoda with analytic spread}), $\overline{I^3}\subseteq I$, but (as we have seen in Remark \ref{bound optimal for F-pure}) $\overline{I^2}\not\subseteq I$.
\end{enumerate}
\end{remark}

\begin{question}
\label{ques: F-closed}
Let $(R,\frakm)$ be a noetherian local ring of prime characteristic $p$. 
\begin{enumerate}
\item\label{ques: power of parameter} Assume each parameter ideal in $R$ is Frobenius closed. Is it true that every power of a parameter ideal is also Frobenius closed?
\item\label{ques: flat local ext of F-closed} Let $R\to R(X)$ be a flat local extension as in Remark \ref{rmk: basic on reduction and analytic spread}. Assume that each parameter ideal in $R$ is Frobenius closed. Does the same hold in $R(X)$?
\end{enumerate}
\end{question}

If Question \ref{ques: F-closed}(\ref{ques: power of parameter}) has an affirmative answer, then (\ref{Skoda with analytic spread}) can be improved to 
\[\overline{I^{2\ell(I)-\height(I)+n}}\subseteq I^n\] 
for each $n\geq 1$. If Question \ref{ques: F-closed}(\ref{ques: flat local ext of F-closed}) has an affirmative answer, then one can show that $\overline{I^{\dim(R)+1}}\subseteq I$ for each ideal $I$ if each parameter ideal in $R$ is Frobenius closed, without assuming the residue field is infinite.

We have a positive answer to Question \ref{ques: F-closed}(\ref{ques: power of parameter}) when the ideal $I$ is generated by a regular sequence.
\begin{proposition}
\label{power of regular sequence F-closed}
Let $(R,\frakm)$ be a noetherian local ring of prime characteristic $p$ and let $g_1,\dots,g_t$ be a regular sequence. Set $I=(g_1,\dots,g_t)$. If $I^F=I$, then $(I^n)^F=I^n$ for each integer $n\geq 1$. 
\end{proposition}
\begin{proof}
We will proceed by an induction on $n$. When $n=1$, there is nothing to prove.

Assume that we have shown $(I^n)^F=I^n$. Let $r$ be an element in $(I^{n+1})^F$. Since $(I^{n+1})^F\subseteq (I^{n})^F=I^n$, it follows that $r\in I^n$. Write $r=\sum c_{\underline{a}}g^{a_1}_1\cdots g^{a_t}_t$ with $\sum_i a_i=n$. Since $r\in (I^{n+1})^F$, there exists an integer $q$ which is a power of $p$ such that
\[r^q=\sum c^q_{\underline{a}}g^{qa_1}_1\cdots g^{qa_t}_t\in (I^{n+1})^{[q]}=(I^{[q]})^{n+1}.\]
Set $h_i=g^q_i$ for $i=1,\dots,t$ and $Q=(h_1,\dots,h_t)=I^{[q]}$. Then $h_1,\dots,h_t$ form a regular sequence by our assumptions. Consider the polynomial 
\[f=\sum c^q_{\underline{a}}Y^{a_1}_1\cdots Y^{a_t}_t\in R[Y_1,\dots,Y_t],\]
which is a homogeneous polynomial of degree $n$. Since $h_1,\dots,h_t$ form a regular sequence and $f(h_1,\dots,h_t)\in Q^{n+1}$, it follows that all coefficients $c^q_{\underline{a}}$ must be in $Q=I^{[q]}$. Hence $c_{\underline{a}}\in I^F=I$; consequently $r\in I^{n+1}$. This completes the proof that $(I^n)^F=I^n$ for all $n\geq 1$.
\end{proof}

The following theorem is an immediate consequence of Theorem \ref{uniform Skoda for I} and Proposition \ref{power of regular sequence F-closed}.
\begin{theorem}
\label{uniform BS for CM with infinite residue field}
Let $R$ be a noetherian Cohen-Macaulay local ring of prime characteristic $p$ with an infinite residue field. Assume that each parameter ideal is Frobenius closed. Then
\[\overline{I^{2\ell(I)-\height(I)+n}}\subseteq I^n\]
for each ideal $I$ with analytic spread $\ell(I)$ and height $\height(I)$ and every integer $n\geq 1$. 

In particular, $\overline{I^{\dim(R)+n}}\subseteq I^n$ for each ideal $I$ and every integer $n\geq 1$.
\end{theorem}

Let $(R,\frakm)$ be a noetherian local ring of prime characteristic. $R$ is called {\it $F$-injective} if the Frobenius action on $H^j_\frakm(R)$ is injective for each $j$. It is well-known that
\begin{enumerate}
\item an $F$-injective ring is reduced, and that 
\item in a Cohen-Macaulay local ring, $F$-injectivity is equivalent to the condition that every ideal generated by a system of parameters being Frobenius closed.
\end{enumerate}

\begin{corollary}
\label{uniform BS for CM}
Let $(R,\frakm)$ be a Cohen-Macaulay $F$-injective local ring of prime characteristic $p$. Then 
\[\overline{I^{\dim(R)+n}}\subseteq I^n\] 
for each ideal $I$ and every integer $n\geq 1$.
\end{corollary}
\begin{proof}
Let $R\to R(X)$ be the flat local extension as in Remark \ref{rmk: basic on reduction and analytic spread}. Then $R(X)$ is also Cohen-Macaulay. Since $R\to R(X)$ is faithfully flat and its closed fiber is a field, $R(X)$ is $F$-injective as well. It follows that each parameter ideal in $R(X)$ is Frobenius closed; consequently $\overline{(IR(X))^{\dim(R)+n}}\subseteq (IR(X))^n$ for each integer $n\geq 1$ by Theorem \ref{uniform BS for CM with infinite residue field}. Since $R(X)$ is faithfully flat over $R$, it follows that
\[\overline{I^{\dim(R)+n}}\subseteq I^n\] 
for each ideal $I$ and every integer $n\geq 1$.
\end{proof}

Recall that a noetherian local ring $(R,\frakm)$ is called \emph{generalized Cohen-Macaulay} if each local cohomology module $H^i_{\frakm}(R)$ has finite length for $i<\dim(R)$. It is proved in \cite[Theorem~1.1]{MaFInjectivity} that if $(R,\frakm)$ is generalized Cohen-Macaulay of prime characteristic then it is $F$-injective if and only if every parameter ideal is Frobenius closed. 

\begin{remark}
Let $(R,\frakm)$ be a noetherian local ring of prime characteristic $p$. Consider the flat local extension $R\to R(X)$ as in Remark \ref{rmk: basic on reduction and analytic spread}. Since $\frakm R(X)$ is the maximal ideal in $R(X)$, one has $H^j_{\frakm R(X)}(R(X))\cong H^j_\frakm(R)\otimes_R R(X)$ for each $j$. Hence $H^j_{\frakm R(X)}(R(X))$ is a finitely generated $R(X)$-module if and only $H^j_\frakm(R)$ is a finitely generated $R$-module. This shows that if $R$ is generalized Cohen-Macaulay then $R(X)$ is also generalized Cohen-Macaulay.
\end{remark}

Following the same strategy as in the proof of Corollary \ref{Skoda for F-pure}, one has the following:

\begin{corollary}
\label{cor: Skoda for F-inj and gCM}
Let $(R,\frakm)$ be an equidimensional and catenary local ring of characteristic $p$. Assume that $R$ is $F$-injective and generalized Cohen-Macaulay. Then
\[\overline{I^{\dim(R)+1}}\subseteq I.\]
for each ideal $I$.
\end{corollary}

We end this section with the following question:
\begin{question}
Let $(R,\frakm)$ be a noetherian local ring of prime characteristic $p$. Assume that $R$ is $F$-injective with an infinite residue field. Let $I$ be an ideal with analytic spread $\ell(I)$ and height $\height(I)$. Is it true that
\[\overline{I^{2\ell(I)-\height(I)+n}}\subseteq I^n\]
for each integer $n\geq 1$?

Or, is it true that $\overline{I^{\dim(R)+n}}\subseteq I^n$ for each ideal $I$ and each integer $n\geq 1$?
\end{question}

\section{Results in characteristic 0}
\label{results char 0}
In this section, we extend some of our results in \S\ref{section: main results in char p} to characteristic 0. We begin with the following counterpart of Corollaries \ref{Skoda for F-pure with generator} and \ref{Skoda for F-pure}.

\begin{theorem}
\label{Skoda for dense F-pure type}
Let $R$ be a local ring essentially of finite type over a field of characteristic 0. Assume that $R$ is of dense $F$-pure type. Then for every ideal $I$ with analytic spread $\ell(I)$ the following holds
\[\overline{I^{\ell(I)+n}}\subseteq I^n,\quad \forall n\geq 1,\]
In particular, $\overline{I^{\dim(R)+n}}\subseteq I^n$ for each integer $n\geq 1$.
\end{theorem}
\begin{proof}
Since residue field of $R$ is infinite, by replacing $I$ with a minimal reduction of $\ell(I)$ generators, we may assume that $I$ can be generated by $\ell=\ell(I)$ elements. Let $z\in \overline{I^{\ell+n}}$ be an arbitrary element. Since the condition $z\in \overline{I^{\ell+n}}$ is equational (given by the integral dependence equation), a combination of the standard reduction mod $p$ argument and Theorem \ref{prop: integral and Frob closure} shows that $z\in I^n$. This completes the proof.
\end{proof}

\begin{remark}
It is shown in \cite{HaraWatanabe} that if $R$ is a normal $\QQ$-Gorenstein local ring essentially of finite type over a field of characteristic 0 and is of dense $F$-pure type then $R$ is log canonical. To the best of our knowledge, it is an open problem whether a log canonical local ring essentially of finite type over a field of characteristic 0 is of dense $F$-pure type. Some partial results are known. 
\end{remark}

In light of Theorem \ref{Skoda for dense F-pure type}, we ask the following:
\begin{question}
\label{ques: Skoda for log canonical}
Let $R$ be a log canonical local ring essentially of finite type over a field of characteristic 0. Is it true that 
\[\overline{I^{\ell(I)+n}}\subseteq I^n,\quad \forall n\geq 1,\]
for every ideal $I$ in $R$?

Or, is it true that $\overline{I^{\dim(R)}+n}\subseteq I^n$ for each ideal $I$ and each integer $n\geq 1$?
\end{question}

Question \ref{ques: Skoda for log canonical} has a positive answer when $\dim(R)=2$ since in this case it is known that `log canonical' implies `dense $F$-pure type' (\cite{Watanabe1988, MehtaSrinivas} and \cite[2.6]{Takagi2013}).

Next we consider the counterpart of Theorem \ref{uniform BS for CM with infinite residue field}. The proof follows the same line of arguments in the one of Theorem \ref{Skoda for dense F-pure type} and is omitted.

\begin{theorem}
\label{uniform BS for dense F-inj type}
Let $R$ be a Cohen-Macaulay local ring essentially of finite type over a field of characteristic 0 that is of dense $F$-injective type. Then 
\[\overline{I^{\dim(R)+n}}\subseteq I^n\] 
for each ideal $I$ and every integer $n\geq 1$.
\end{theorem}

Conjecturally `dense $F$-injective type' is equivalent to `Du Bois'; it is known that `dense $F$-injective type' implies `Du Bois' (\cite{Schwede2019}). Hence it is natural to ask:
\begin{question}
Let $R$ be a Du Bois local ring essentially of finite type over a field of characteristic 0. Let $I$ be an ideal with analytic spread $\ell(I)$ and height $\height(I)$. Is it true that
\[\overline{I^{2\ell(I)-\height(I)+n}}\subseteq I^n\]
for each integer $n\geq 1$?

Or, is it true that $\overline{I^{\dim(R)+n}}\subseteq I^n$ for each ideal $I$ and each integer $n\geq 1$?
\end{question}

\end{document}